\documentclass[11pt]{amsart}
\usepackage{amsmath, amsrefs, amsthm, amssymb}
\usepackage{graphicx}
\usepackage{caption}
\usepackage{subcaption}
\usepackage{pdfsync}	
\usepackage{keyval} 
\usepackage{float}				
\newcommand\R{{\mathbb{R}}}

\newcommand\N{{\mathbb{N}}}
\newcommand\C{{\mathbb{C}}}

\newcommand\Z{{\mathbb{Z}}}
\newcommand\T{{\mathbb{T}}}
\newcommand\A{{\mathbb{A}}}

\DeclareMathOperator\sign{\text{sign}}
\newtheorem{theorem}{Theorem}[section]
\newtheorem{proposition}{Proposition}[section]
\newtheorem{lemma}[theorem]{Lemma}

\theoremstyle{definition}
\newtheorem{definition}[theorem]{Definition}

\theoremstyle{remark}
\newtheorem{remark}[theorem]{Remark}
\numberwithin{equation}{section}

\def\R{\mathbb{R}}
\def\C{\mathbb{C}}

\begin{document}
\title[$b$-family of Equations]
{Blow-up for the  $b$-family of equations}
\author{ Fernando Cortez}
\address{Escuela Polit\'ecnica Nacional, Departamento de Matem\'atica,  Facultad de Ciencias,
Ladr\'on de Guevara E11-253, Quito-Ecuador, P.O. Box 17-01-2759}
\email{manuel.cortez@epn.edu.ec}
\begin{abstract}
In this paper we consider the  $b$-family of equations on the torus $u_t- u_{txx}+ (b+1) u u_x=b u_x u_{xx} + u u_{xxx}$, which for appropriate values  of $b$ reduces to  well-known models,  such as the Camassa-Holm  equation or the  Degasperis-Procesi equation. We establish a local-in-space blow-up criterion. 
\end{abstract}
\maketitle
\section{Introduction}
The bi-Hamiltonian structure of
certain evolution equations leads to  various remarkable features such as infinitely many symmetries and conserved quantities, and in some cases to the exact solvability of these equations  \cite{Magri,Olver}. Examples include the  KdV equation \cite{Cau} and the Benjamin-Ono equation \cite{Bona}. 
Years later, R. Camassa and D. Holm \cite{CH1}  in their studies of completely integrable dispersive shallow water equation tackled the following  equation, 
\begin{equation}
\label{ch2}
u_t  + k u_x -u_{xxt}+3uu_x =uu_{xxx} + 2 u_{x} u_{xx},   \ \ \ x \in \mathbb{R}, \  t>0. \tag{C-H} 
\end{equation}
where $u$ can be interpreted as a horizontal velocity of the water at a certain depth and $k$ as the dispersion parameter. The equation \eqref{ch2}  also has been derived independently by B.  Fuchssteiner and A. Fokas in \cite{Fuch}.
When $k=0$ (dispersionless case), the equation \eqref{ch2}  possess soliton 
solutions peaked at their crest
(often named peakons) \cite{CH1,CH2,Bra13}. Equation (\ref{ch2}) is obtained by using an asymptotic expansion directly in the Hamiltonian for Euler's equation in the shallow water regime. Like the KdV equation, the Camassa-Holm equation \eqref{ch2} describes the unidirectional propagation of waves at the surface of shallow water under the influence of gravity \cite{Cau,CH1}. 
The equation \eqref{ch2} is physically relevant as it also 
describes  the nonlinear dispersive waves in compressible hyperelastic rods \cite{Bra12,Bra13,Dai98}.
  It is  convenient to rewrite the Cauchy problem   associated with the dispersionless  case of (\ref{ch2}) in the following weak form:
\begin{equation}
 \label{c3}
  \begin{cases}
   u_t+u u_x+\partial_xp*\left(u^2 + \frac{ u_{x}^{2}}{2}\right)=0, &x\in\A,\quad t>0,\\
   u(0,x)=u_0(x) & x\in\A, 
  \end{cases}
  \end{equation}
  where $p(x)$ is the fundamental solution of the operator $1-\partial_{x}^{2}$ in $\A$. If $\A= \R$,  we refer (\ref{c3}) as the 
  non-periodic Camassa-Holm equation and  $p=\frac{1}{2} e^{- \left | x \right |}$, $x\in \R$ in this case. 
  If otherwise that  $\A= \T= \R/ \Z$  is unit circle, we refer (\ref{c3}) as the periodic Camassa-Holm equation,
  and $p= \frac{\cosh(x-\left[x\right]-\frac{1}{2})}{2\sinh\left(\frac{1}{2}\right)}$
  in this case.
 It is know that both the  non-periodic and periodic  Camassa-Holm equations are locally well-posed (in the sense of Hadamard)  
  in the Sobolev space $H^{s}$, with $s>\frac{3}{2}$. See  \cite{RB,Dan2001, Lel}.
    There is an abundance of the literature about the issue of the finite time blowup
    (see~\cite{McKean04,Bra13, Gui,LiOl, ConEschActa,ACon00,Bra1,Bra12}) 
     and the related issue of the global existence of strong solutions (\cite{ConEschActa,Gui,Mol}). 
     
On the other hand, Degasperis and Procesi \cite{DP1}, in their search of new integrability properties  inside a wide class of equations,
were led to consider the following  integrable equation:
\begin{equation}
\label{Dp1}
u_t - u_{txx}+4 uu_x= 3 u_{x} u_{xx} + u u_{xxx}.  \tag{D-P} 
\end{equation}
As before, it is convenient to rewrite the Cauchy problem, using the same notations
\begin{equation}
\label{DP2}
\begin{cases}
   u_t+u u_x+\partial_xp*\left(\frac{3}{2}u^2 \right)=0, &x\in\A,\quad t>0,\\
   u(0,x)=u_0(x), & x\in\A.  
  \end{cases}
  \end{equation}
A few years later, equation~\eqref{DP2} as been proved to be relevant in shallow water dynamics,  see  \cite{Dul,ConLann,Cons2011}. 
Both the Camassa--Holm equation and the the Degasperis-Procesi equation \eqref{Dp1} possess a
bi-Hamiltonian structure (see\cite{DP1}).
The local well-posedness in $H^{s}$, with $s>\frac{3}{2}$ for the Cauchy non periodic problem was elaborated in \cite{Yin1},  and  \cite{Yin2}  for the Cauchy periodic problem.  With respect to blow-up criteria on the line we refer to \cite{DP1,Liu,Zh,Gu} and, for the unit torus, to \cite{Yin2,Gu}. For  the existence  globally  of the solution, see \cite{Yin1,Gu, Liu}.
Despite sharing some properties with the Camassa-Holm equation, the Degasperis-Procesi has its own peculiarities.
A specific feature 
is that \eqref{Dp1} admits, beside peakons (i.e., soliton solutions of the form $u(t,x)=c e^{-\left | x-ct \right |}$, $c>0$) 
also shock peakon solitons  (i.e., solutions at the form $u=\frac{1}{t+k}\sign(x) e^{-\left | x-ct \right |}$, $k>0$).
For more details see \cite{Es1,Lun,He}.
After these premises,  we will now focus on  the Cauchy problem for the spatially periodic $b$-family of equations:
\begin{equation}
 \label{beq1}
  \begin{cases}
   u_t- u_{txx}+ (b+1) u u_x=b u_x u_{xx} + u u_{xxx} ,   &x\in\T,\quad t>0,\\
   u(0,x)=u_0(x), & x\in\T,
  \end{cases}
\end{equation}
 where $\T$ is the unit torus.
 Here $b$  is a real parameter, and $u(x,t)$ stands for a horizontal velocity.
The $b$-family of equations can be derived as the family of asymptotically equivalent shallow water wave equations  that emerges at quadratic-order accuracy for any $b\neq 1$ by an appropriate Kodama transformation  \cite{DP1,Dul}. Again, when $b=2$ and $b=3$  (\ref{beq1}) became \eqref{ch2} and \eqref{Dp1} respectively. These values  are the only values for  which (\ref{beq1}) is completely integrable.
The Cauchy problem for the $b$- family of equations is locally well posed in the Sobolev space $H^{s}$ for any $s>\frac{3}{2}$, \cite{Liu,Es2,Sa,Og}.
 
In  \cite{Che} it is  proved  that the solution map of the $b$-family of equations is Holder continuous as a map from bounded sets of $H^{s}(\mathbb{R})$, $s>\frac{3}{2}$ with the $H^{r}(\mathbb{R})$ $(0\leq r<s)$ topology, to $C([0,T],H^{r}(\mathbb{R}))$.
 While that in \cite{Chris}, the authors proved  that the solution map is not uniformly continuous. Their proof relies on a construction of smooth periodic traveling waves with small amplitude.
 J. Escher and J. Seiler \cite{Es2} showed that the periodic $b$-family of equations can be realized as Euler equation on the Lie group $\text{Diff}^{\infty}(\T)$  of all smooth and orientation-preserving diffeomorphisms on the unit torus, if $b=2$ (C-H equation).
The global existence theory of the solution of (\ref{beq1}) is discussed in \cite{Sa,Es2,Liu,Yin2}. 
In this paper we rather focus on blow-up criteria as well in estimates about the lifespan of the solutions.
The blowup problem for the $b$-family of equations has been already treated, e.g. in  \cite{LiOl,Sa,Es2,Og,Gu,Zh1}: 
in these references the condition on the initial datum $u_0$ leading to the blowup typically involves the computation of some global quantities (the Sobolev norm $\|u_0\|_{H^1}$, or some
other integral expressions of~$u_0$).
Motivated by the recent paper~\cite{Bra13} (where earlier blowup results for the Camassa--Holm equations were unified in a single theorem) we address the more subtle problem of finding a \emph{local-in-space} blowup criterion for the $b$-family of equations,
i.e., a blowup condition involving only the properties of $u_0$ in a neighborhood of a single point $x_0\in\T$.

Loosely, the contribution of this paper can be stated as follows: if the parameter $b$ belongs to a suitable range (including the
physically relevant cases $b=2$ and $b=3$),  then  then there exists a constant $\beta_b>0$ such that if   
\begin{equation*} 
  u'_0(x_0)  < - \beta_b  \left | u_0(x_0) \right |,
 \end{equation*}
 in at least one point  $x_0 \in \T$, then the solution arising from $u_0\in H^{s}(\T)$ $(s>\frac{3}{2})$ must blow-up in finite time.

This paper is organized as follows.
In the next section we start by introducing the relevant notations and function spaces,  recalling a few basic results. 
Then we precisely state and prove our main theorem.
An important part of our work will be devoted to the computations of sharp bounds for the constant $\beta_b$ and the lifespan of the solution.
The smallest $b>1$ to which our main theorem applies is computed numerically in the last part of the paper.

 \section{Blow-up for the  periodic $b$-family of equations}
 
It is convenient to rewrite the periodic Cauchy problem (\ref{beq1}) in the following weak form (see \cite{Sa}):

\begin{equation}
 \label{beq3}
  \begin{cases}
   u_t+u u_x+\partial_xp*\Bigl[\frac{b}{2}u^2 + \left(\frac{3-b}{2}\right) u_{x}^{2}\Bigr]=0, &x\in\T,\quad t>0,\\
   u(0,x)=u_0(x), & x\in\T \\
 u(t,x)=u(t,x+1) \ \ \ \ t\geq 0, 
  \end{cases}
\end{equation}
 where 
\begin{equation}
p(x)=\frac{\cosh(x-\left[x\right]-\frac{1}{2})}{2\sinh\left(\frac{1}{2}\right)}, 
\end{equation}
 is the fundamental solution of the operator $1-\partial^2_x$ and $\left[ \cdot\right]$ stands for the integer part of $x\in\R$.
If  $u\in C([0,T),H^s(\T))\cap C^{1}((0,T^{*}),H^{s-1}(\T))$, with $s>\frac{3}{2}$ satisfies (\ref{beq3}) then we call $u$ a strong solution to (\ref{beq3}). If $u$ is a strong solution on $[0,T)$ for every $T>0$, then is called global strong solution of (\ref{beq3}).

If $u_0\in H^{s}(\T)$, $s >\frac{3}{2}$, an application of Kato's method~\cite{Kat} leads to the following local well-posedness
result:

\begin{theorem}[See \cite{Sa}]
 \label{th:Sa}
For any constant  $b$, given $u_0 \in H^{s}(\T)$, $s>\frac{3}{2}$, then there exists a maximal time 
$T^{*}= T^{*}(\left \| u_0 \right \|_{H^{s}})>0$ and  a unique strong solution $u$ to (\ref{beq3}), such that 
\begin{equation}
 u=u(\cdot,u_0) \in C([0,T^{*}),H^s(\T))\cap C^{1}([0,T^{*}),H^{s-1}(\T)).
 \end{equation}
Moreover, the solution depends continuously on the initial data, i.e. the mapping $u_0\mapsto u(\cdot,u_0):H^{s}(T) \rightarrow C([0,T^{*});H^{s}(T)) \cap C^{1}([0,T^{*});H^{s-1}(T))$ is continuous. 
\end{theorem}
\begin{remark}
The maximal lifespan of the solution  in Theorem \ref{th:Sa} may be chosen independently of $s$ in the following sense: 
If  $u=u(\cdot,u_0) \in C([0,T^{*}),H^s(\T))\cap C^{1}([0,T^{*}),H^{s-1}(\T))$  to (\ref{beq3}) and $u_0 \in H^{s'}(\T)$ for some $s'\neq s$,  $s'>\frac{3}{2}$, then $u=u(\cdot,u_0) \in C([0,T^{*}),H^{s'}(\T))\cap C^{1}([0,T^{*}),H^{s'-1}(\T))$ and with same $T^{*}$. In particular, if $u_0\in \cap_{s\geq 0} \   H^{s}$, then $u \in C([0,T^{*}),H^{\infty}(\T))$.
See \cite{Sa,Og}.
\end{remark}
Moreover, by using the Theorem \ref{th:Sa} and energy estimates, the following precise blow-up scenario of the  solution to (\ref{beq3}) can be obtained.
\begin{theorem}[See \cite{Og,Sa}] 
Assume $b\in \R$ and $u_0 \in H^{s}(\T)$, $s>\frac{3}{2}$. Then blow up of the strong solution $u=u(\cdot,u_0)$ in finite time  occurs if only if
\begin{equation}
\lim_{t\to T^{*}} \inf \{ (2b-1) \inf_{x\in \R} [ u_x(t,x) ] \} = - \infty
\end{equation}
\end{theorem}
Before presenting our contribution, we will review a few known blow-up theorems with respect to (\ref{beq3}).
\begin{theorem}[See \cite{Sa}]
Let $\frac{5}{3} < \ b\leq 3 \ \  \mbox{and} \displaystyle\int_{\T} u_{x}^{3} (0) \, dx < 0$. Assume that $u_0 \in H^{s}(\T)$,  $s >\frac{3}{2}$, $ u_0  \not\equiv  0$, and the corresponding solution $u(t)$ of (\ref{beq3})  has a zero for any time $t\geq 0$. Then, the solution $u(t)$ of the  equation \eqref{beq3}  blows-up finite time.
\end{theorem} 
The next blow-up  theorem uses the fact that if $u(t,x)$ is a solution to (\ref{beq3})  with initial datum $u_0$, then  $-u(t,-x)$ is also a solution to (\ref{beq3}) with initial datum $-u_0(-x)$. Hence due to the uniqueness of the solutions, the solution  to (\ref{beq3}) is odd as soon as the initial datum $u_0(x)$ is odd. 
\begin{theorem} [See \cite{Og}]
Let $1< b \leq 3.$ Assume that $u_0\in H^{s}(\T)$, $s>3$ is odd and $u^{'}_{0}(0) < 0.$ Then the corresponding solution to Eq \eqref{beq3} blows-up in finite time.
\end{theorem}
\subsection*{Notations}
For any real $\beta$, let us consider the $1$-periodic function
\begin{equation}
w(x) = p(x)+ \beta \partial_xp(x)
\end{equation}
where $p$ is the kernel introduced in (\ref{beq3}) and $\partial_x p$ denotes the distributional derivative on $\R$, that agrees in this case with the classical i.e pointwise derivative on $\R 	\setminus \Z$. 
Notice that  the non-negativity condition $w\geq 0$ is equivalent to the inequality 
$\cosh(1/2) \geq \pm \beta \sinh(1/2)$, i.e., to the condition
\begin{eqnarray*}
-\frac{e+1}{e-1} \leq \beta \leq  \frac{e+1}{e-1}.
\end{eqnarray*}
Throughout this section, we will work  under the above condition on $\beta$.
Let us now introduce the following weighted Sobolev space:
\begin{eqnarray}
\label{beta1}
E_{\beta}&=& \{u\in L^{1}_{loc}(0,1): \left \| u \right \|_{E_{\beta}}^{2} = \displaystyle\int_{0}^{1} w(x) (u^{2}+ u_{x}^{2})(x) \, dx < \infty \},
\end{eqnarray}
where the derivative is understood in the distributional sense. Notice that   $E_{\beta}$ agrees with the classical Sobolev space $H^{1}(0,1)$ when $\left | \beta \right |< \frac{e+1}{e-1}$, as in this case $w$ is bounded and bounded away from $0$, and the two norms $\left \| \cdot \right\|_{E_{\beta}}$ and $\left \| \cdot \right \|_{H^{1}}$ are equivalent. 
The situation is different for $\beta= \pm \frac{e+1}{e-1}$ as $E_{\beta}$ is strictly larger that $H^{1}(0,1)$ in this case. 
Indeed, we have
\begin{eqnarray}
\label{omega}
w(x)= \frac{2 e}{(e-1)^{2}} \sinh(x), \ \ \ \  x\in(0,1), \ \ \ \left(\mbox{if}  \ \beta=\textstyle\frac{e+1}{e-1}\right);
\end{eqnarray}   
The elements of $E_{(e+1)/(e-1)}$, after modification on a set of measure zero, 
are continuous on $(0,1]$, but may be unbounded for $x \rightarrow 0^{+}$ (for instance, $\left | \log(x/2) \right |^{1/3} \in E_{(e+1)/(e-1)}$). In the same way, 
\begin{eqnarray}
\label{omega2}
w(x)= \frac{2 e}{(e-1)^{2}} \sinh(1-x), \ \ \ \  x\in(0,1), \ \ \ (\mbox{if}  \ \beta=- \textstyle\frac{e+1}{e-1});
\end{eqnarray}  
after modification on a set of measure zero, the elements of $E_{-(e+1)/(e-1)}$ are continuous on $[0,1)$, but may be unbounded for $x \rightarrow 1^{-}$.
\newline
Let us now introduce the closed subspace $E_{\beta,0}$ of $E_{\beta}$ defined as the closure of $C^{\infty}_{c}(0,1)$ in $E_{\beta}$.
The elements of $E_{\beta,0}$ satisfy  the weighted Poincar\'e inequality below:
\begin{lemma}
For all  $\left | \beta \right | \leq \frac{e+1}{e-1}$, there exists a constant  $C>0$ such that 
\begin{equation}
\label{poin}
\forall v\in E_{\beta,0},
\quad\displaystyle\int_{0}^{1} w(x)  \ v^{2}(x) \, dx \leq C  \displaystyle\int_{0}^{1} w(x) \  v_x ^{2}(x) \, dx.
\end{equation}
\end{lemma}
\begin{proof}
This demonstration is found in \cite{Bra1}. 
\end{proof}
We need some notations.
\begin{definition}
For any real constant $b\neq 1$ and $\beta$, let $J(b,\beta) \geq -\infty$, be defined by
\begin{equation}
J(b,\beta)= \inf \left\{\displaystyle\int_{0}^{1}(p + \beta \partial_x p) \left(\frac{b}{2} u^2 +\left(\frac{3-b}{2}\right) u_x^2\right) \, dx ; \  u\in H^{1}(0,1),  \ u(0)=u(1)= 1\right\}
\end{equation} 
and
\begin{equation}
\label{beta}
\beta_b=\inf \left\{\beta >0:  \ \  \beta^{2}+\frac{2}{\left | b-1 \right |} \left(J(b,\beta)-\frac{b}{2}\right)  \geq  0\right\} .
\end{equation}
\end{definition}
Notice that a priori   $0\leq \beta_{b}  \leq +\infty$, as the set on the right-hand side could be empty.

\subsection*{Main results}

Let us now formalize the goal of this paper.
\begin{theorem}
\label{teo1}
Let $ b \in]1,3] $ be such that  $\beta_{b}$ is finite. Let $u_0 \in H^{s}(\T) $ be with $s > \frac{3}{2}$ and assume that there exists $x_0 \in \T$, such that 
\begin{equation}
u'_{0}(x_0) < - \beta_b  \left | u_0(x_0) \right |.     
 \end{equation}
then the corresponding solution $u$ of (\ref{beq3}) in $C([0,T^{*}), H^{s}(\T)) \cap C^{1}([0,T^{*}), H^{s-1}(\T))$ arising from $u_0$   blows up  in finite time. 
Moreover, the maximal time $T^{*}$ verifies
\begin{eqnarray}
T^{*} \leq \frac{2}{(b-1)\sqrt{(u'_{0}(x_0))^{2}- \beta_{b}^{2} u_{0}^{2}(x_0)}}.
\end{eqnarray}
\end{theorem}
\begin{remark}
Notice that the Theorem \ref{teo1} relies on the condition that $\beta_{b}$ is finite. 
In section 2, we will prove that one indeed has $\beta_b<+\infty$, as soon as $b$ is outside a very small neighborhood of~$1$.
On the other hand, as we will see later on, for $1<b<1.0012\ldots$, $\beta_b=+\infty$ and Theorem~\ref{teo1} does not apply
in such range.
\end{remark}
For the proof of  Theorem \ref{teo1}, we need the following propositions.
\begin{proposition}
We have 
\begin{equation}
\label{func1}
 J(b,\beta) > -\infty  \Leftrightarrow\left \{ \begin{matrix} \left | \beta \right | \leq \frac{e+1}{e-1}, &
 \\
\\  b \leq 3,
\\
\\ \frac{b}{3-b}  > -\frac{1}{C_{\beta}}, & \end{matrix}\right.
\end{equation}
where $C_{\beta}>0$ is the best Poincar\'e constant in inequality  (\ref{poin}).
\end{proposition}
\begin{proof}
Putting $u=v+1$ and observing that $\int_{0}^{1} w(x)\,dx = 1$, we see that 
\begin{equation}
\label{j1}
J(b,\beta)= \frac{b}{2}  + \inf \{ T(v) : v\in H_{0}^{1}(0,1) \}, 
\end{equation}
where 
\begin{equation}
\label{j2}
T(v) = \displaystyle\int_{0}^{1} w(x) \left( \frac{b}{2}(v^{2}+2v)+\left(\frac{3-b}{2}\right)v_{x}^{2}\right) (x) \, dx. 
\end{equation}
Assume that $J(b,\beta) > -\infty$. In order to  show $\left | \beta \right |\leq \frac{e+1}{e-1}$, we refer to  the proof of proposition 3.3. in \cite{Bra1}. In order to prove $b\leq3$, we consider $\left | \beta \right |\leq \frac{e+1}{e-1}$ and
\begin{equation}
 u_{n}(x)= 1 + \frac{1}{2} \sin(n2\pi x)  \ \ \Rightarrow  \ \  u'_{n}(x)= n \pi \cos(n2\pi x).
\end{equation}
For each  $n\in \N$  $u_n \in H^{1} (0,1)$, $u_n(1)=u_n(0) =1$. Thus there is a constant $c_1>0$ independent of $n$, such that  
\begin{eqnarray*}
\forall n\in \N  \  \   \   0 \leq \frac{b}{2} \displaystyle\int_{0}^{1} w(x) u_n^{2}(x) \, dx  \leq c_1, 
\end{eqnarray*}
and
\begin{equation*}
\frac{3-b}{2}  \displaystyle\int_{0}^{1} w(x) (u'_{n})^{2} (x) \, dx \ \rightarrow  \ - \infty,
 \end{equation*}
because $b>3$ and then $J(b,\beta)=-\infty$.
In order to prove the third inequality, we only have to treat the case  $b < 0$.  Applying the inequality
\begin{equation}
\displaystyle\int_{0}^{1} w(x) \left( \frac{b}{2}(n^{2} v^{2}+2nv)+\left(\frac{3-b}{2}\right) n^{2}v_{x}^{2}\right) (x) \, dx \geq J(b,\beta)- \frac{b}{2},
\end{equation}
valid for all $v\in H_{0}^{1}(0,1)$ and all $n\in \mathbb{N}$ and letting $n \rightarrow \infty$, we get 
\begin{eqnarray*}
\displaystyle\int_{0}^{1} w(x) \left( \frac{b}{2} v^{2}+ \left( \frac{3-b}{2}\right)v_{x}^{2}\right) (x)\, dx \geq 0. 
\end{eqnarray*}
We deduce:
\begin{equation*}
\displaystyle\int_{0}^{1} w(x) v^{2}(x) \,dx \leq - \frac{3-b}{b} \displaystyle\int_{0}^{1} w(x) v_{x}^{2}(x) \,dx.  
\end{equation*}
Then we get  $\frac{b}{3-b}\geq -\frac{1}{C_\beta}$.
But the equality case $\frac{b}{3-b}= -\frac{1}{C_\beta}$ can be excluded, as otherwise
we could find a sequence $v_n$ such that $((b/2)\int_0^1 \omega v_n^2)/((3-b)\int_0^1 \omega (v_n)_x^2)$ converges to $1$ and such that $\int b\omega v_n\to-\infty$: for such a sequence we have $T(v_n)\sim \int_0^1 b\omega v_n\to-\infty$, contradicting
the assumption $J(b,\beta)>-\infty$.

Conversely, assume that  $\left | \beta \right | \le \frac{e+1}{e-1}$. By the weighted Poincair\'e inequality (\ref{poin}), we can consider an equivalent norm in $E_{\beta,0}$: 
\begin{equation}
\left \| v \right \|_{E_{\beta,0}} = \displaystyle\int_{0}^{1} w(x) v_{x} (x) \, dx.
 \end{equation}
 Since $\frac{b}{3-b}>-\frac{1}{C_\beta}$,  the symmetric bilinear form  
\begin{equation}
B(u,v) =  \displaystyle\int_{0}^{1} w(x) \left( \frac{b}{2} u v + \left( \frac{3-b}{2}\right) u_{x} v_{x}\right) (x)\, dx, 
\end{equation}
 is coercive  on the Hilbert space $E_{\beta,0}$. Applying the Lax-Milgram theorem yields the existence and uniqueness of  a minimizer $\hat{v} \in E_{\beta,0}$ for the functional $T$.
 But $ H^{1}_{0}(0,1) \subset E_{\beta,0}$, so in particular, we get $J(b,\beta) >-\infty$. Moreover, if $\left |\beta \right | < \frac{e+1}{e-1}$, then recalling $E_{\beta,0}= H^{1}_{0}(0,1)$ we see that $J(b,\beta)$ is in fact a minimun, achieved at $\hat{u}= 1+\hat{v} \in H^{1}(0,1)$.       
\end{proof}
The next lemma provides some useful information about $J(b,\beta)$.
\begin{lemma}
\label{c2}
The function $(b,\beta) \mapsto J(b,\beta) \in \R \ \cup \{-\infty\} $ defined for all $(b,\beta)\in \R^{2}$ is concave with respect  to each one of its variables, and is even with respect  to the variable $\beta.$ Also for all $b\in\R$ and $\left | \beta \right | \leq\frac{e+1}{e-1}$,  $ \  \ -\infty \leq J(b,\frac{e+1}{e-1}) \leq J(b,\beta) \leq J(b,0) \leq \frac{b}{2}$.  
\end{lemma}
\begin{proof}
The proof is similar to  that of the proposition 3.4. in \cite{Bra1}
\end{proof}
The next lemma motivates the introduction of quantity the $J(b,\beta)$ in relation with the  $b$-family of equations.
\begin{proposition}
\label{propb}
Let $(\alpha,\beta) \in \R^{2}$ and $u\in H^{1}(\T)$, we get
\begin{equation*}
\forall x\in \T,  \ \ (p + \beta \partial_x p) * \left(\frac{b}{2} u^{2} +\left(\frac{3-b}{2} \right) u_{x}^2) \right)(x) \geq J(b,\beta)  \ u^{2}(x).
\end{equation*}
\end{proposition}
\begin{proof}
Let $\alpha=\alpha(b,\beta)$ be some  constant. Because of  the invariance under translation, we get that the inequality
\begin{equation}
\label{prop1}
 (p + \beta \partial_x p) * \left(\frac{b}{2} u^{2} +\left(\frac{3-b}{2} \right) u_{x}^2) \right)(x) \geq \alpha \ u^{2}(x),
\end{equation} 
holds true for all $u\in H^{1}(\T)$ and all $x\in \T$ if and only if the inequality
\begin{equation}
  (p + \beta \partial_x p) * \left(\frac{b}{2} u^{2} +\left(\frac{3-b}{2} \right) u_{x}^2) \right)(1) \geq \alpha \ u^{2}(1),
\end{equation} 
holds true for all $u\in H^{1}(\T)$. But on the interval $]0,1[$, $(p + \beta \partial_x p) (1-x)=(p - \beta \partial_x p)(x)$. Then we get
\begin{equation}
  (p + \beta \partial_x p) * \left(\frac{b}{2} u^{2} +\left(\frac{3-b}{2} \right) u_{x}^2) \right)(1)  =\displaystyle\int_{0}^{1} (p - \beta \partial_x p)  \left(\frac{b}{2} u^{2} +\left(\frac{3-b}{2} \right) u_{x}^2 \right)(x) \, dx.
\end{equation}
Normalizing to obtain $u(1)=1$, we get that the best constant  $\alpha$ in inequality (\ref{prop1}) satisfies $\alpha=J(b,-\beta)=J(b,\beta)$.
\end{proof}
The next proposition provides a first lower bound estimate of  $J(b,\beta)$, when $b \in [-1,3]$.
\begin{proposition}
\label{propi}
Let  $-1\leq b \leq 3$ and $\left | \beta \right |\leq \frac{e+1}{e-1}$. Then, if $u\in H^{1}(0,1)$ such that $u(1)=u(0)$, we get  
\begin{eqnarray*}
 \label{propj}
(p \pm \beta \partial_x p) * \left(\frac{b}{2} u^{2} +\left(\frac{3-b}{2} \right) u_{x}^2 \right) \geq \ \begin{cases}
 \delta_{b} \  u^{2} ,   \ \mbox{if}    \  \   \left | \beta  \right | \leq 1  \\
\\
\frac{\delta_b}{2}[(e+1) - \left |\beta \right | (e-1)] u^2, \   \mbox{if} \    \ 1\leq \left | \beta  \right| \leq \frac{e+1}{e-1}, 
\end{cases}
\end{eqnarray*} 
where 
\begin{equation}
\label{del}
\delta_b = \frac{\sqrt{3-b}}{4}\left(\sqrt{3(1+b)}-\sqrt{3-b} \right).
\end{equation}
\end{proposition}
\begin{remark}
Notice that $\delta_b\ge0$ if and only if  for $0\leq b \leq 3$.
\end{remark}
\begin{proof}
It is sufficient to consider the case $0 \leq \beta \leq \frac{e+1}{e-1}$.  We make the convolution estimates  for $(p+\beta \partial_x p)$, the convolution estimates for $(p-\beta \partial_x p)$ being similar. First observe that:
\begin{equation}
\forall x \in \R  \ \  p(x)= \frac{e^{x-\frac{1}{2}-[x]}}{4 \sinh \frac{1}{2}} + \frac{e^{-x+\frac{1}{2}+[x]}}{4 \sinh \frac{1}{2}}  = p_1(x) + p_2(x).  
\end{equation} 
 We start with the estimate of $p_1*(a^2 u^2 +u_x^{2})(1)$, with $a\in \R$ to be determined later. We get
\begin{eqnarray*}
p_1 * ( a^{2}u^{2}+u^{2}_x)(1)&=& \frac{1}{4\sinh(\frac{1}{2})}\displaystyle\int_{0}^{1} e^{\frac{1}{2}-\xi} (a^{2}u^{2}+u_{x}^{2})(\xi) \,d\xi   	\\
&\geq& \frac{-a}{4\sinh(\frac{1}{2})}\displaystyle\int_{0}^{1} e^{\frac{1}{2}-\xi} (2u u_x)(\xi)  \, d\xi \\
&=&   \frac{-a}{4\sinh(\frac{1}{2})} (e^{\frac{-1}{2}}-e^{\frac{1}{2}}) u^{2}(1) - \frac{1}{4\sinh(\frac{1}{2})}\displaystyle\int_{0}^{1} e^{\frac{1}{2}-\xi} a u^{2} \,d\xi  \\
&=& \frac{a}{2} u^{2}(1) - p_1 *(au^{2})(1).
\end{eqnarray*}
Hence 
\begin{equation*}
p_1*((a^{2}+a)u^{2}+u_{x}^{2}) (1) \geq \frac{a}{2}u^{2}(1), 
\end{equation*}
and because of the invariance under translations, we get 
\begin{equation}
\label{p1}
p_1*((a^{2}+a)u^{2}+u_{x}^{2}) \geq \frac{a}{2}u^{2}. 
\end{equation}
Similarily:
\begin{eqnarray*}
\label{p2}
p_2 * ( a^{2}u^{2}+u^{2}_x)(1)&=& \frac{1}{4\sinh(\frac{1}{2})}\displaystyle\int_{0}^{1} e^{\xi-\frac{1}{2}} (a^{2}u^{2}+u_{x}^{2})(\xi) \,d\xi   	\\
&\geq& \frac{a}{4\sinh(\frac{1}{2})}\displaystyle\int_{0}^{1} e^{\xi -\frac{1}{2}} (2u u_x)(\xi)  \, d\xi \\
&=&   \frac{a}{4\sinh(\frac{1}{2})} (e^{\frac{1}{2}}-e^{\frac{-1}{2}}) u^{2}(1) - \frac{1}{4\sinh(\frac{1}{2})}\displaystyle\int_{0}^{1} e^{\xi - \frac{1}{2}} a u^{2} \,d\xi  \\
&=& \frac{a}{2} u^{2}(1) - p_2*(au^{2})(1).
\end{eqnarray*}
Hence, again using the invariance under translations, we get
\begin{equation}
\label{p1}
p_2*((a^{2}+a)u^{2}+u_{x}^{2}) \geq \frac{a}{2}u^{2}. 
\end{equation}
Choose $a$ such that $a^{2} + a = \frac{b}{3-b}$. This is indeed possible if 
$ -1\leq b <3$ (if  $b=3$, the proposition is trivial and there is nothing to prove). We get: 
\begin{eqnarray}
\label{p2}
p_1  * \left(\frac{b}{2} u^{2} +\left(\frac{3-b}{2} \right) u_{x}^2 \right)&\geq& \frac{\delta_b}{2}u^{2}, \\
p_2  * \left(\frac{b}{2} u^{2} +\left(\frac{3-b}{2} \right) u_{x}^2 \right)&\geq& \frac{\delta_b}{2}u^{2}.
\end{eqnarray}
Now, from the identity $p= p_1 + p_2$ and $\partial_x p = p_1 - p_2$, that holds both in the distributional and in the a.e. 
pointwise sense,  we get
\begin{equation}
\label{p3}
p+ \beta \partial_x p = (1+\beta) p_1 + (1-\beta) p_2.
 \end{equation}
 If  $0 \leq \beta \leq 1$, then from (\ref{p2}) and (\ref{p3}), we deduce
 \begin{equation}
 \label{p4}
(p+ \beta \partial_x p)* \left(\frac{b}{2} u^{2} + \left(\frac{3-b}{2} u_{x}^2\right) \right) \geq [(1+\beta)+(1-\beta)] \frac{\delta_b}{2} u^{2}= \delta_b u^{2}. 
 \end{equation}
 We proved as follows. From
\begin{equation}
p_2(x) \leq e \ p_1(x), \quad\forall x\in(0,1),
\end{equation}
we get, for $1\leq \beta \leq  \frac{e+1}{e-1}$ 
\begin{equation}
\begin{split}
p+ \beta \partial_{x}p &= (1+\beta) p_1 -(\beta -1) p_2, \\
&\geq [(e+1)-\beta(e-1)] p_1.
\end{split}
\end{equation}
We deduce, using (\ref{p2}):
\begin{eqnarray}
\forall  \  1 \leq \beta \leq \textstyle\frac{e+1}{e-1}, \quad
(p+\beta\partial_xp) \left( \frac{b}{2} u^{2}+ \left(\frac{3-b}{2} u_{x}^2\right) \right) \geq [(e+1)-\beta(e-1)]
 \frac{\delta_b}{2} u^{2}.
 \end{eqnarray}
\end{proof}
\begin{remark}
If  $-1 \leq b \leq 3$,  then it follows by the preceding proposition that
$\left | \beta \right |\leq 1$, then  $J(b,\beta)\geq \delta_b$, and if $1 \leq \left | \beta \right | \leq \frac{e+1}{e-1}$ then $J(b,\beta)\geq\frac{\delta_b}{2}[(e+1) - \left |\beta \right | (e-1)]$. 
\end{remark}
 \begin{proof}[Proof of Theorem~\ref{teo1}]
By the well-posedness result in $H^s(\T)$, with $s>3/2$, the density of $H^3(\T)$ in $H^s(\T)$ and a simple approximation
argument, we only  need to prove Theorem~\ref{teo1} assuming  $u_0\in H^3(\T)$. 
We thus obtain a unique  solution of (\ref{beq3}), defined in some nontrivial interval $[0,T[$, and such that $u\in C([0,T[,H^{3}(\T)) \cap C^{1}([0,T[,H^{2}(\T))$. The starting point is the analysis of the flow map $q(t,x)$  of (\ref{beq3})
\begin{equation}
 \label{flowb}
 \begin{cases}
 q_t(t,x)=u(t,q(t,x)) &x\in\T,\quad t \in [0,T^{*}),\\
   q(0,x)=x, & x\in\T .
 \end{cases}
\end{equation}
As $u\in  C^{1}([0,T[,H^{2}(\T))$, we can see that $u$ and $u_x$ are continuous on $[0,T[ \times \T$ and $x\mapsto u(t,x) $ is Lipschitz, uniformly with respect to $t$ in any compact time interval in $[0,T[$. Then the flow map $q(t,x)$ is well defined by (\ref{flowb})  in the time interval $[0,T[$ and $q\in C^{1}([0,T[ \times \R, \R)$.
Differentiating  (\ref{beq3}) with respect to the $x$ variable and applying the identity  $\partial_{x}^{2}p * f = p*f  - f$, we get:
\begin{eqnarray*}
u_{t x} + u u_{xx}&=&\frac{b}{2}u^{2}  - \left(\frac{b-1}{2}\right) u_{x}^{2} -  p*\Bigl[ \frac{b}{2} u^2 + \left(\frac{3-b}{2}\right) u_{x}^{2}\Bigr] .
\end{eqnarray*}
Let us introduce the two $C^{1}$ functions of the time variable depending on $\beta$. The constant $\beta$, will be chosen later on
\begin{eqnarray*}
f(t)= \left(-u_x + \beta u\right) (t,q(t,x_0))  \ \ \ \mbox{and} \ \ \ g(t)= -\left(u_x + \beta u\right)(t,q(t,x_0)).
\end{eqnarray*}
Using \eqref{flowb} and differentiating with respect  to $t$, we get
\begin{eqnarray*}
\frac{d f}{dt}(t)&=& [(-u_{tx}-uu_{xx})+ \beta(u_t+u u_x)](t,q(t,x_0)) \\ 
&=&  -\frac{b}{2} u^{2}+ \left(\frac{b-1}{2} \right)u_{x}^{2} + (p-\beta \partial_xp) *\Bigl[ \frac{b}{2} u^2 + \left(\frac{3-b}{2}\right) u_{x}^{2}\Bigr] (t,q(t,x_0)),
\end{eqnarray*}
and 
\begin{eqnarray*}
\frac{d g}{dt}(t)&=& [(-u_{tx}-uu_{xx})- \beta(u_t+u u_x)](t,q(t,x_0)) \\ 
&=&  -\frac{b}{2} u^{2}+ \left(\frac{b-1}{2} \right)u_{x}^{2} + (p + \beta \partial_xp) *\Bigl[ \frac{b}{2} u^2 + \left(\frac{3-b}{2}\right) u_{x}^{2}\Bigr] (t,q(t,x_0)).
\end{eqnarray*}  
Let us first consider $b\in ]1,3]$.
Recall that we work under the condition $\beta_b <\infty$.
By the definition of $\beta_b$ (\ref{beta}) we deduce that there exists $\beta \geq 0$ such that
\begin{eqnarray}
\label{pr}
\beta^{2} \geq  \frac{2}{ b-1}  \left(\frac{b}{2}-J(b,\beta)\right).
\end{eqnarray}
Applying the convolution estimate of (\ref{propb}) and the fact that $J(b,\beta)=J(b,-\beta)$, we get
\begin{eqnarray*}
\frac{d f}{dt}(t)&\geq&  \left( \frac{b-1}{2}\right) u_{x}^{2} +\left( J(b,-\beta)-\frac{b}{2}\right) u^{2}  (t,q(t,x_0)) \\
&\geq&    \frac{b-1}{2}  \ (u_{x}^{2} -\beta^{2} u^{2})  \  (t,q(t,x_0)) \\
&=& \frac{b-1}{2} [f(t)g(t)]
\end{eqnarray*}
In the same way, 
 \begin{eqnarray*}
\frac{d g}{dt}(t)&\geq&  \left( \frac{b-1}{2}\right) u_{x}^{2} +\left( J(b,\beta)-\frac{b}{2}\right) u^{2}  (t,q(t,x_0)) \\
&\geq&    \frac{b-1}{2}  \  (u_{x}^{2} -\beta^{2} u^{2}) \   (t,q(t,x_0)) \\
&=& \frac{b-1}{2} [f(t)g(t)].
\end{eqnarray*}
The assumption $u'_{0}(x_0) < -\beta_b\left | u_0(x_0) \right |$ guarantees that we may choose $\beta$ satisfying (\ref{pr}) with $\beta-\beta_b >0$  small enough so that 
\begin{displaymath}
 u'_{0}(x_0) < -\beta\left | u_0(x_0) \right |.
\end{displaymath}
For such a choice of $\beta$ we have $f(0)>0  \ $ and $g(0) >0$.

We now make use of the following result:
\begin{lemma}[See \cite{Bra1}]
Let $0<T^{*}  \leq \infty$ and  $f,g \in C^{1}([0,T^{*}[,\R)$ be such that, for some constant $c>0$ and all $t \in [0, T^{*}[$,
\  \begin{eqnarray*}
\frac{d f}{dt}(t)&\geq&  c f(t) g(t) \\
 \frac{d g}{dt}(t) &\geq&   c f(t) g(t).
\end{eqnarray*}
If $f(0)>0$ and $g(0) >0$, then 
\begin{displaymath}
T^{*} \leq \frac{1}{c\sqrt{f(0)g(0)}}.
\end{displaymath}
\end{lemma}
The blow-up of $u$ then follows immediately from our previous estimates  applying the above lemma.
\end{proof}
\section{estimates of $\beta_b$}

Theorem \ref{teo1} is meaningful only if $b\in (1,3]$ is such that  $\beta_b < \infty$.
We recall here that~$\beta_b$ is defined by Eq.~\eqref{beta}:
\begin{equation*}
\beta_b=\inf \left\{\beta >0:  \ \  \beta^{2}+\frac{2}{\left | b-1 \right |} \left(J(b,\beta)-\frac{b}{2}\right)  \geq  0\right\} .
\end{equation*}

Next, we propose three lower bound estimates
for the convolution term
\[(p \pm \beta \partial_x p) * \left(\frac{b}{2} u^{2} +\left(\frac{3-b}{2} \right) u_{x}^2 \right),\]
or ---what is equivalent, owing to~Proposition~\ref{propb}--- three lower bound estimates for $J(b,\beta)$).
Such estimates will allow us to determinate sufficient conditions on~$b\in(1,3]$ in order to $\beta_b$ to be finite
and will provide upper bounds for~$\beta_b$.

Estimate~1 and Estimate~2 below are presented mainly for pedagogical purposes, as they are self-contained. But 
these two estimates will be later on improved by Estimate~3, which is more technical and deeply relies on a few  involved computations made in~\cite{Bra1}.   
We point out however that Estimate~1 suffices to claim that~Theorem~\ref{teo1} is not vacuous. 
\subsection{Estimate 1}
  Let $0 \leq \beta \leq \frac{e+1}{e-1}$ and  $1 < b  \leq 3$. We start considering the obvious estimate
\begin{equation*}
(p \pm \beta \partial_x p) * \left(\frac{b}{2} u^{2} +\left(\frac{3-b}{2} \right) u_{x}^2 \right) \geq  0.
\end{equation*}
 Thanks to definition \eqref{beta}, we see that a sufficient condition on~$b$  which entails $\beta_b < \infty$, is the existence of a constant~$\beta$ satisfying
\begin{equation}
\label{es1}
\sqrt{\frac{b}{b-1}} \le
 \beta  \leq \frac{e+1}{e-1}.
\end{equation}
This holds when
$b\ge  \frac{(e+1)^2}{4e}\equiv\alpha$.
In this case, the corresponding bound for $\beta_b$ is
\begin{equation}
\label{ron}
\beta_b \le \sqrt{\frac{b}{b-1}}<+\infty, \qquad \text{for}  \ \ \textstyle\frac{(e+1)^2}{4e}\le b \le3.
\end{equation} 
(See Figure  \ref{cac1}).
\subsection{Estimate 2}
Proposition \ref{propi} provides a better sufficient condition ensuring that $\beta_b <  +\infty$. Namely: 
\begin{eqnarray}
\label{es3}
 \exists  \  0\leq \beta \leq 1 \ \ \ \mbox{such that} \ \ \ \beta^{2}+ \frac{2}{b-1} \left(\delta_b - \frac{b}{2}\right)  \geq 0, 
\end{eqnarray}
or
\begin{eqnarray}
\label{es4}
 \exists  \  1\leq \beta \leq \frac{e+1}{e-1} \ \ \ \mbox{such that} \ \ \ \beta^{2}+ \frac{2}{b-1}\left( [(e+1)-\beta(e-1)] \frac{\delta_b}{2}- \frac{b}{2} \right)  \geq 0,
\end{eqnarray}
where $\delta_b$ is as (\ref{del}).
The study of the function $b \mapsto   \sqrt{\frac{2}{b-1}\left(\frac{b}{2}-\delta_b \right)}$ in the interval $(1,3]$ however reveals that condition~\eqref{es3} is satisfied only for $b=2$. We have $\delta_2= \frac{1}{2}$ and so $\beta=1$.
The corresponding estimate for $\beta_2$ is then $\beta_2\le1$.
This situation corresponds to the Camassa--Holm equation. We thus recover the result in~\cite{Bra13}. 
(See Figure \ref{cac2}.)  
On the other hand, solving~\eqref{es4} is possible if and only if the largest real zero $\phi(b)$ of the quadratic polynomial
$
 \beta\mapsto P_{b}(\beta) = \beta^{2} + \beta  \ \delta_b   \left(\frac{e+1}{b-1}\right) + \left(\delta_b \left(\frac{e+1}{b-1} \right)- \frac{b}{b-1}\right)
 $
is inside the interval $[1,\frac{e+1}{e-1}]$.

A simple computation shows that this is indeed the case when $\alpha\le b\le 3$.
Here $\alpha=\frac{(e+1)^2}{4e}$ is the same as in Estimate~1.
For $\alpha\le b\le 3$, now we get the bound
 \begin{equation}
 \label{vodka}
 \beta_b \leq \phi(b)<+\infty,  
 \end{equation}
 that considerably improves our earlier estimate (\ref{ron}). See Figure~\ref{cac2}
\begin{figure}
\centering
\begin{subfigure}[]{0.6\textwidth}
\includegraphics[width=\textwidth]{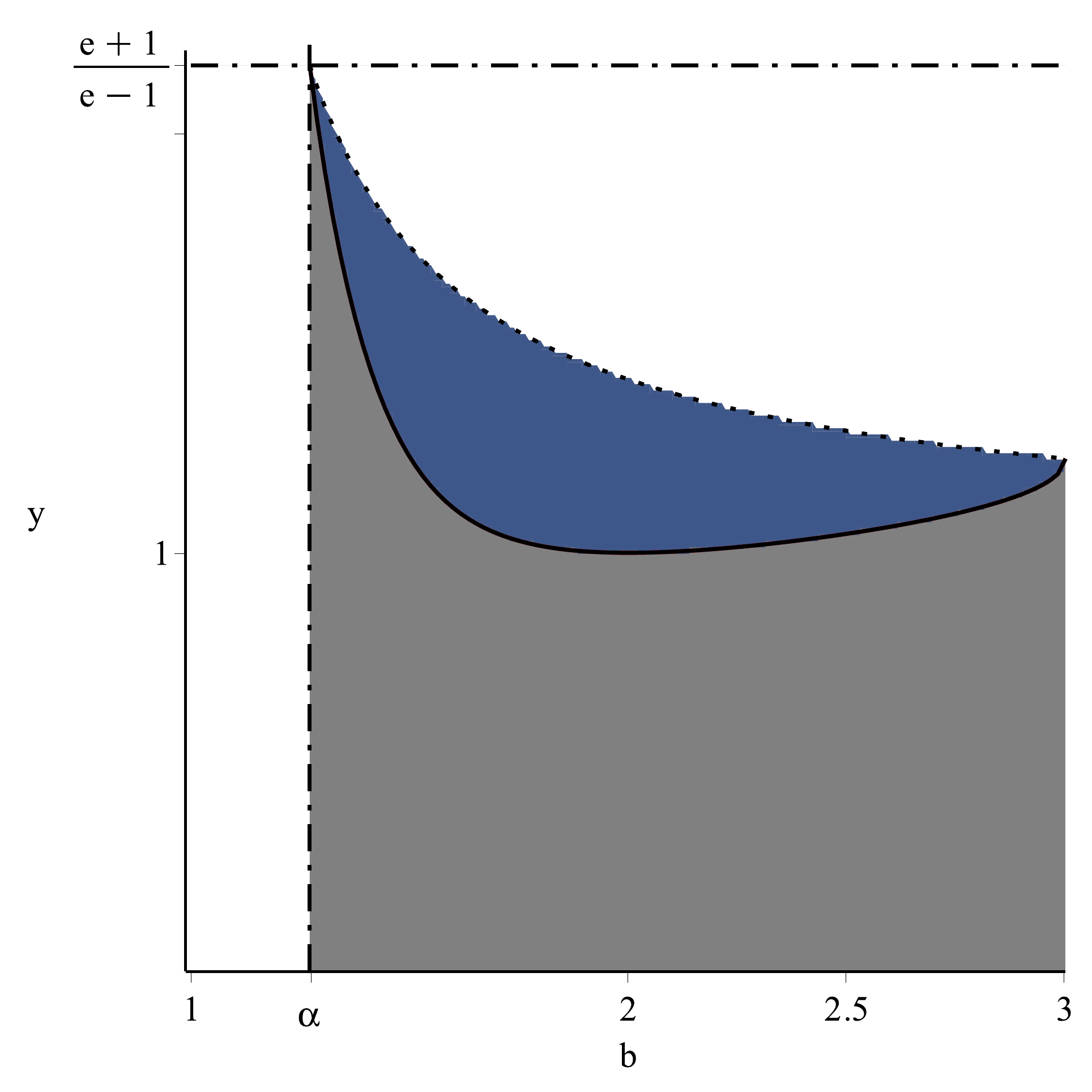}
\caption{ The plot of the function $b \mapsto   \sqrt{\frac{b}{b-1}}$, providing the bound  (\ref{ron}). 
  The upper-bound estimate of $\beta_b$ given by Eq.(\ref{ron}), showing that 
  Theorem~\ref{teo1} applies for $b\in [\alpha,3]$,
  where $ \alpha= \frac{(e+1)^2}{4e}$ (blue and gray region).}
  \label{cac1}
\caption{The function $b \mapsto \phi(b)$, providing the bound  (\ref{vodka}). The   upper-bound estimates of $\beta_b$  given by Eq.(\ref{vodka}) and the Theorem \ref{teo1} are valid inside the interval $[\alpha,3]$ (grey region).}
\label{cac2}
\end{subfigure}  \hspace{3.38pt}
\caption{First and Second estimate of  $\beta_b$.}
\label{figu2}
\end{figure}
\subsection{Estimate 3}

This part relies  on the properties of $J(b,\beta)$ which are described in  Lemma \ref{c2} and the computations made in~\cite{Bra1}

Let $I(\alpha,\beta)$ as in~\cite[Section~2]{Bra1}.
For $b\in (1,3]$, and $ \left | \beta \right| \leq \frac{e+1}{e-1}$, the relation between $I$ and $J$ is the following:
\begin{eqnarray*}
 \label{func}
J(b,\beta)= \begin{cases}
 \frac{3-b}{2}  \ I\left(\frac{b}{3-b}, \beta\right),   &  \mbox{if }   b\neq 3 \\
  \frac{3}{2}  \ \inf\left\{   \displaystyle\int_{0}^{1}  w(x) \  u^{2} \, dx ; \  u\in H^{1}(0,1),  \ u(0)=u(1)= 1 \right\}, &   \mbox{if } b=3. 
\end{cases}
\end{eqnarray*}
where $I(\alpha,\beta)$ is as in \cite{Bra1}. 
If $b\neq 3$, borrowing the computation made in \cite{Bra1}, we get

\begin{eqnarray*}
J\left(b,\frac{e+1}{e-1}\right)&=& \frac{3-b}{2}   \ I  \left(\frac{b}{3-b},\frac{e+1}{e-1} \right) \\
 &=&  \frac{3-b}{4e} \left( e+1\right)^{2}  \frac{P'_{\upsilon(b)}}{P_{\upsilon(b)}} (\cosh 1)
\end{eqnarray*}
where
\begin{eqnarray*}
\upsilon(b)= -\frac{1}{2} +\frac{1}{2} \cdot \sqrt{1+4\cdot \left( \frac{b}{3-b}\right)} \in\{z\in\C: \Im(z) \geq 0 \} .
\end{eqnarray*}
and $P_{\upsilon(b)}$ is Legendre function of the first kind, of the degree $\upsilon(b)$, arising when solving the Euler--Lagrange 
equation associated with the minimization problem of~$I(\alpha,\frac{e+1}{e-1})$.
The reason for considering here the limit case $\beta=\frac{e+1}{e-1}$ is twofold: on one hand, in this case the weight function
has a simpler expression, namely 
$w(x)$ becomes in this case
 \begin{eqnarray*}
 w(x)= p(x) + \textstyle\frac{e+1}{e-1}\,\partial_x p(x)= \textstyle\frac{2 e}{(e-1)^{2}}  \sinh x, \qquad x\in(0,1);
 \end{eqnarray*}
 this allow to reduce the Euler-Lagrange equation to a linear second order ordinary differential equation of Legendre type.
 See~\cite{Bra1} for more details.
On the other hand, by Lemma~\ref{c2}, we have  $J(b,\beta)\ge J\left(b,\frac{e+1} {e-1}\right)$ 
for all $0\le\beta\le \frac{e+1}{e-1}$.

Now, for $0\le\beta\le \frac{e+1}{e-1}$, we have
\begin{equation}
 \label{oio}
 \begin{split}
\beta^2+\frac{2}{b-1}\biggl( J(b,\beta)-\frac b2\biggr)
&\ge \beta^2 +
\frac{2}{b-1}
\left( 
\frac{3-b}{4e} \left( e+1\right)^{2}  \frac{P'_{\upsilon(b)}}{P_{\upsilon(b)}} (\cosh 1) -\frac{b}{2}
\right).
\end{split}
\end{equation}
Computing the Legendre function shows that the right hand-side of the above expression is nonnegative when $\gamma\le b\le3$,
with $\gamma\approx1.012$. See Figure~\ref{fi}.
Therefore, in the range $b\in[\gamma,3]$ we have $\beta_b<+\infty$
\begin{equation}
\label{wisky}
\beta_b \leq  \sqrt{\frac{2}{b-1}
\left(\frac{b}{2} - \frac{3-b}{4e} \left( e+1\right)^{2}  \frac{P'_{\upsilon(b)}}{P_{\upsilon(b)}} (\cosh 1) \right)},
\qquad \text{for $\gamma\le b\le 3$},
\end{equation}
and Theorem~\ref{teo1} applies in such range.
\begin{figure}
\centering
\includegraphics[width=0.58\textwidth]{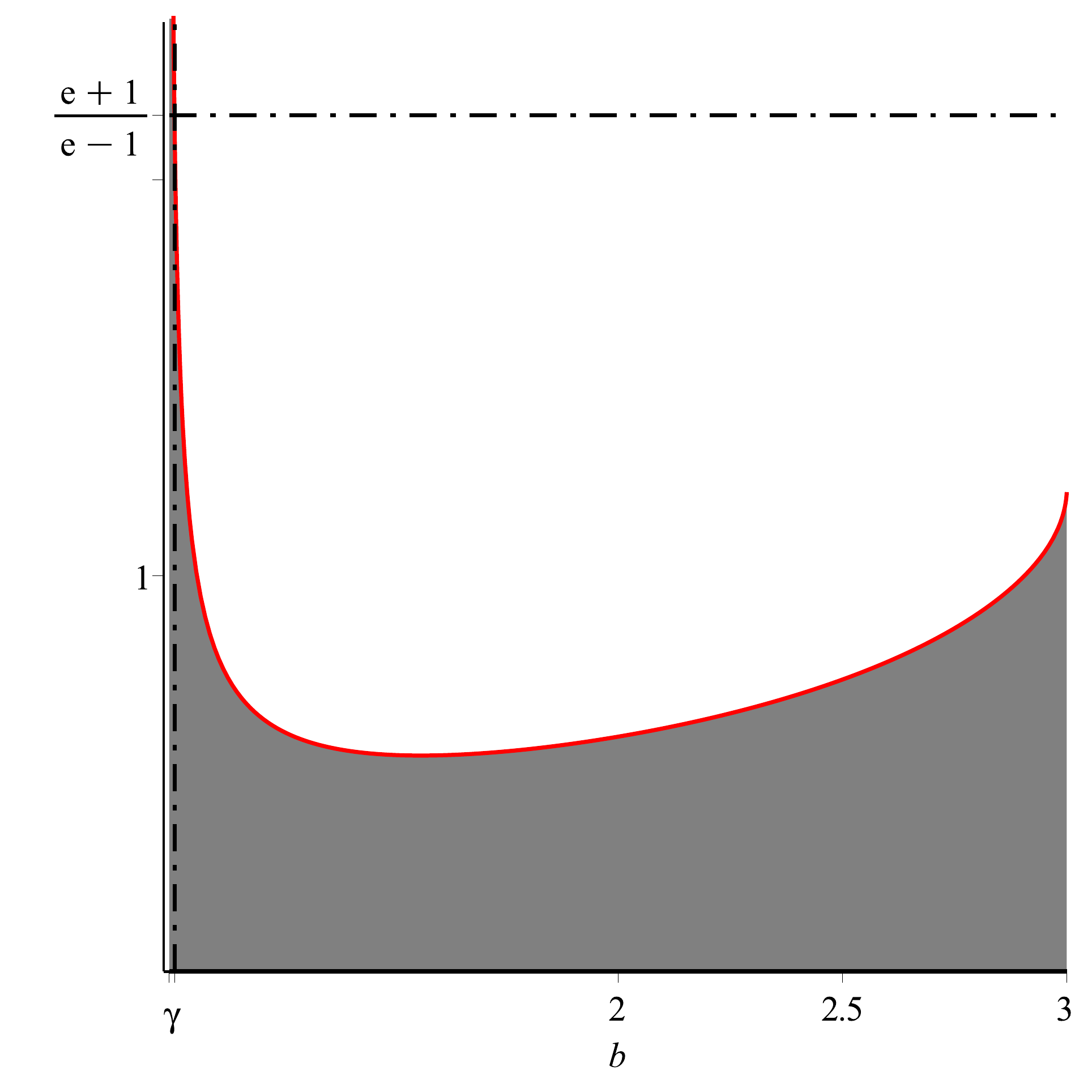}
  \caption{ The function 
  $b\mapsto
   \sqrt{\frac{2}{b-1}
\left(\frac{b}{2} - \frac{3-b}{4e} \left( e+1\right)^{2}  \frac{P'_{\upsilon(b)}}{P_{\upsilon(b)}} (\cosh 1) \right)}$,
  providing the bound  (\ref{wisky}). 
  The  upper-bound estimates of $\beta_{b}$ given by Eq.(\ref{wisky}) and the Theorem \ref{teo1}
  are valid inside the interval $[\gamma,3]$ (grey region)}
 \label{fi}
\end{figure}
\subsection{Numerical Analysis  of $\beta_b$}
In this last part we compute numerically $\beta_b$. We need first to compute numerically $J(\beta,b)$.
Recall that 
\begin{equation*}
\label{j1}
J(b,\beta)= \frac{b}{2}  + \inf \{ T(v) : v\in H_{0}^{1}(0,1) \}, 
\end{equation*}
where 
\begin{equation}
\label{j2}
T(v) = \displaystyle\int_{0}^{1} w(x) \left( \frac{b}{2}(v^{2}+2v)+\left(\frac{3-b}{2}\right)v_{x}^{2}\right) (x) \, dx. 
\end{equation}
The Euler-Lagrange equation associated with  the above minimization problem is
\begin{equation}
\label{EL}
(3-b) w(x) v_{xx} + (3-b) w_{x} v_{x} - b w v - b w=0.
\end{equation}
Let  $\bar{v}$ be the solution such that $\bar{v}(0)=\bar{v}(1)=0$, i.e $\bar{v}$ is the minimiser :
\begin{equation}
\label{EE1}
J(b,\beta)= \frac{b}{2} + \displaystyle\int_{0}^{1} w(x) \left( \frac{b}{2}\bar{v}^{2} + b \bar{v} + \left(\frac{3-b}{2}\right) \bar{v}_{x}^{2} \right)(x) \, dx.
\end{equation}
On the other hand, multiplying  (\ref{EL}) by $\bar{v}$ and  integrating with respect to the spatial variable, we get 
\begin{eqnarray*}
\displaystyle\int_{0}^{1} (3-b) w \bar{v}_{xx} \bar{v} \, dx + \displaystyle\int_{0}^{1} (3-b) w_x \bar{v} \bar{v}_{x} \, dx - \displaystyle\int_{0}^{1} b w (\bar{v}^{2}+\bar{v}) \, dx =0.
\end{eqnarray*}
 Integrating by parts $\displaystyle\int_{0}^{1} (3-b) w_x \bar{v} \bar{v}_{x}\,dx$, and using that $\bar{v}(0)=\bar{v}(1)=0$, we get  
\begin{eqnarray*}
&&\displaystyle\int_{0}^{1} (3-b) w \bar{v}_{x}^{2} + \displaystyle\int_{0}^{1} b (w \bar{v}^{2}+\bar{v}) \, dx =0  \\
&& \displaystyle\int_{0}^{1} b w\bar{v} \, dx  = \displaystyle\int_{0}^{1} w \left( b(\bar{v}^{2}+2\bar{v})+(3-b)\bar{v}_x^{2}\right)\, dx.
\end{eqnarray*}  
Thus, using $\int_{0}^{1}w\, dx=1$ and $(3-b) (w v_{xx} + w_{x} v_{x}) = bw (v + 1)$, we get 
\begin{eqnarray*}
J(b,\beta)&=& \frac{b}{2} + \displaystyle\int_{0}^{1} \frac{b}{2} w \bar{v} \,dx \\
&=& \frac{3-b}{2} \displaystyle\int_{0}^{1}  [w \bar{v}_{x}]_x \,dx \\
&=& \frac{3-b}{2} \left[(w \bar{v}_x)(1^{-})-(w \bar{v}_x)(0^{+}) \right].
\end{eqnarray*}
The above solution $\bar v$ of the minimization problem, depending on the parameters $b$ and $\beta$, cannot be computed
analytically, but it it can be computed numerically with the standard numerical schemes for linear ODEs,  with an arbitrary good precision.
This allow to compute numerically the above function $J(b,\beta)$.
This being done, a simple algorithm allows to compute numerically the quantity $\beta_b$ (with an arbirary good precision).
Such numerical computations illustrate that in fact $\beta_b<+\infty$ for $1.0012\ldots \le b\le3$, which is (slightly !) better
than the range $1.012\le b\le3$ obtained via Estimate~3.
The actual value of $\beta_b$ is actually slightly smaller than its upper bound computed in~\eqref{wisky}.
See Figure~\ref{fbeta} and~\ref{dernif}.
\begin{figure}
\centering
\includegraphics[width=0.58\textwidth]{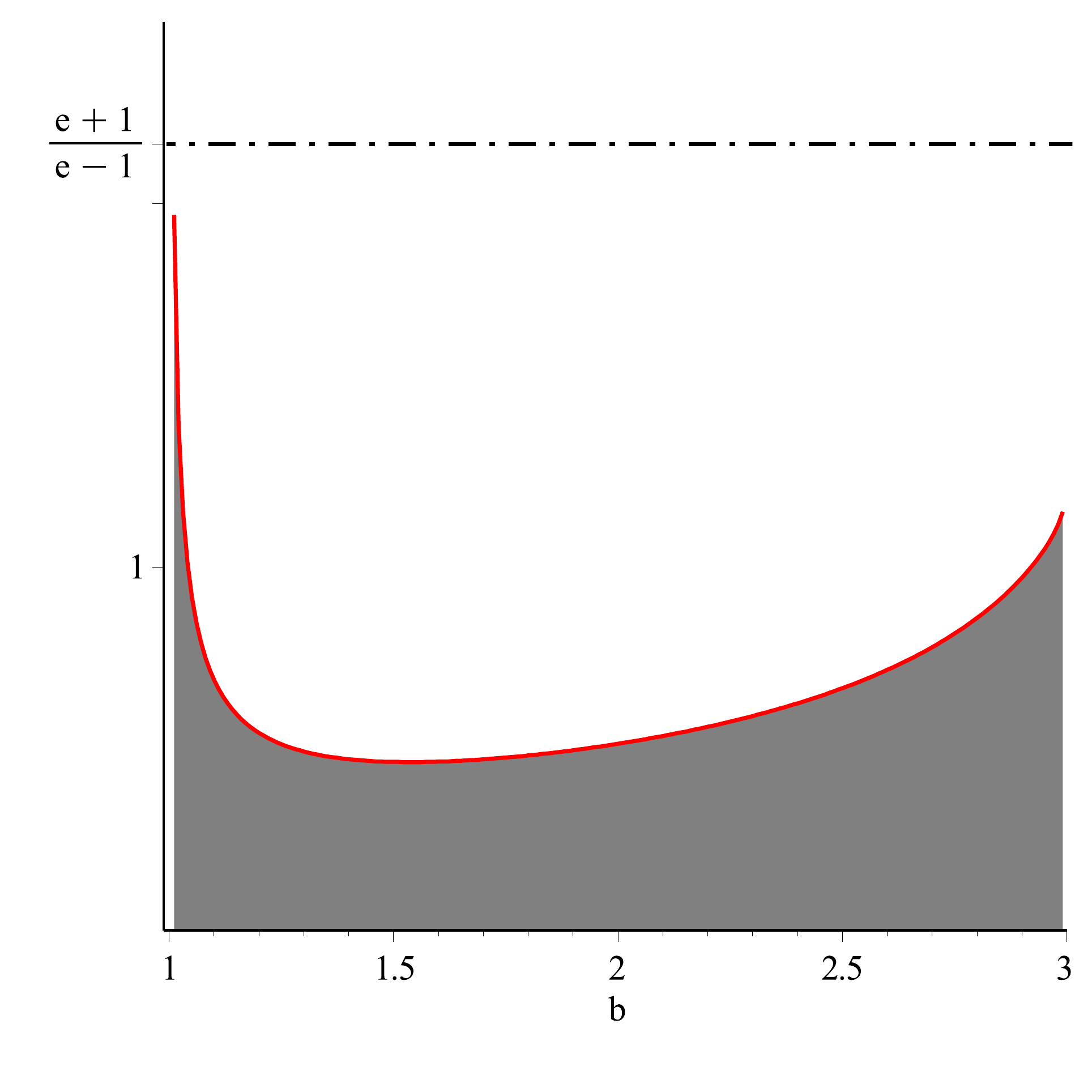}
  \caption{  The plot of  the function $b \mapsto \beta_b $. This numerical approach of $\beta_b$, allows us to say: if  $3\geq b  \geq \alpha_0 \approx 1.0012$, then the Theorem \ref{teo1} is valid (gray region).}
 \label{fbeta}
\end{figure}
We summarize in the last picture all our previous estimates and numerical approximate of $\beta_b$.  
\begin{figure}
\centering
\includegraphics[width=0.58\textwidth]{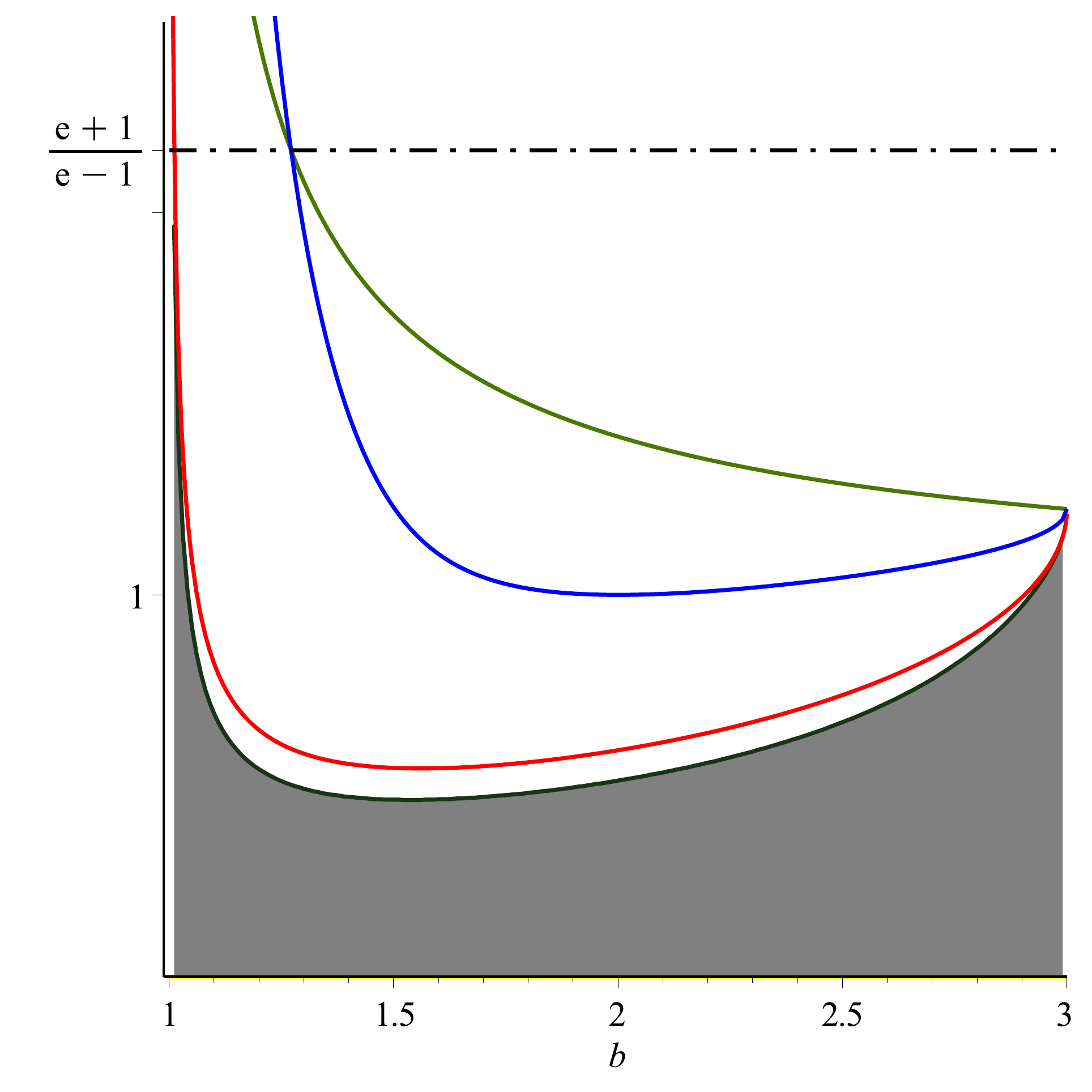}
  \caption{ In this plot we can see the different estimates that we  have worked out (green curve first estimate, blue curve second estimate and red curve third estimates), as well as the  numerical approach of $\beta_b$.} 
\label{dernif}
  \end{figure}
\section*{Acknowledgement}
Also, I would like to thank \textit{Escuela Polit\'ecnica Nacional del Ecuador (EPN)} where most of this paper was written. I appreciate the comfortable and relaxing place it is. \\
This work is supported by  the \textit{Secretar\'ia de Educaci\'on Superior, Ciencia, Tecnolog\'ia e Innovaci\'on del Ecuador  (SENESCYT).} 


\begin{thebibliography}{} 
\bib{Bona}{article}{
   author={T.B. Benjamin},
   title={Internal waves of permanent form in fluids of great depth },
   journal={J. Fluid Mech.},
   volume={29},
   date={1967},
   pages={559--562.}}
   
   \bib{Bra13}{article}{
   author={Brandolese L.},
   title={Local-in-space criteria for blowup in shallow water and dispersive rod equations},
   journal={Comm. Math. Phys.},
   volume={330},
   pages={401--414}
   date={2014},
} 
\bib{Bra12}{article}{
  title={Blowup issues for a class of nonlinear dispersive wave equations},
  author={ Brandolese L.},
  author={Cortez M.F.},
  journal={Journal of Differential Equations},
  volume={256},
  number={12},
  pages={3981--3998},
  year={2014},
  publisher={Elsevier}
}
\bib{Bra1}{article}{
   author={Brandolese L.},
   author={Cortez M.F.},
   title={On permanent and breaking waves in hyperelastic rods and rings},
   journal={J. Funct. Anal.}, 
   volume={266},
   date={2014},
   pages={6954-6987}
} 
\bib{CH1}{article}{
   author={Camassa R.},
   author={ Holm L.},
   title={An integrable shallow water equation with peaked solitons},
   journal={Physical Review Letters},
   date={1993}
   volume={71}
   pages={1661-1664}}
   \bib{CH2}{article}{
   author={Camassa R.},
   author={Holm L.},
   author={Hyman J.M.},
   title={A new integrable shallow water equation},
   journal={Adv. Appl. Mech},
   date={1994}
   volume={31}
   pages={1-31}}
\bib{Cau}{article}{
   title={A new hierarchy of Korteweg-de Vries equations},
  author={Caudrey P.J.},
   author={Dodd R.K.}, 
   author={Gibbon J.D.},
  journal={Proceedings of the Royal Society of London. A. Mathematical and Physical Sciences},
  volume={351},
  number={1666},
  pages={407--422},
  year={1976},
  publisher={The Royal Society}}
\bib{Che}{article}{
year={2013},
journal={Mathematische Annalen},
volume={357},
number={4},
title={The H{\"o}der continuity of the solution map to the $b$-family of equations in weak topology},
publisher={Springer Berlin Heidelberg},
author={Chen M.  R.},  
 author={Liu Y.},
 author={ Zhang P.},
pages={1245-1289},}
\bib{Og}{article}{
  title={On the Cauchy problem for the periodic $b$-family of equations and of the non-uniform continuity of Degasperis--Procesi equation},
  author={Christov O.},
  author={Hakkaev S. },
  journal={Journal of Mathematical Analysis and Applications},
  volume={360},
  number={1},
  pages={47--56},
  year={2009},
  publisher={Elsevier}
}
	\bib{Chris}{article}{  
title={Non-uniform continuity of periodic Holm-Staley $b$-family of equations},
author={Christov O.},
author={Hakkaev S.},
author={Iliev I.},
journal={Nonlinear Analysis: Theory, Methods $\&$ Applications},
volume={75},
number={13},
pages={4821 - 4838},
year={2012},
}
  \bib{ConEschActa}{article}{
   author={Constantin A.},
   author={Escher J.},
   title={Wave breaking for nonlinear nonlocal shallow water equations},
   journal={Acta Math.},
   volume={181},
   date={1998},
   number={2},
   pages={229--243},
} 
  \bib{ACon00}{article}{
   author={Constantin  A.},
   title={Existence of permanent and breaking waves for a shallow water equation: A geometric approach},
   journal={Ann. Inst. Fourier (Grenoble)},
   volume={50},
   year={2000},
   pages={321--362},
    }
\bib{Mol}{article}{
   author={Constantin A.},
   author={Molinet L.}, 
   title={ Global weak solutions for a shallow water equation.}
   journal={Comm. Math. Phys.}, 
   volume={211},
   date={2000},
      pages={45--61},  
}
\bib{ConLann}{article}{
   author={Constantin A.},
   author={Lannes D.},
   title={The hydrodynamical relevance of the Camassa-Holm and Degasperis-Procesi equations},
   journal={Arch. Ration. Mech. Anal},
   volume={192},
   date={2009},
   number={2},
   pages={165--186},
}

\bib{Cons2011}{article}{
 author={Constantin, Adrian},
  title={Nonlinear Water Waves with Applications to WaveCurrent
Interactions and Tsunamis.},
journal={CBMS-NSF Regional Conference Series in Applied Mathematics, SIAM}
  volume={81},
  year={2011},
  pages={321 pp}
}
\bib{Dai98}{article}{
   author={Dai, H.-H.},
   title={Model equations for nonlinear dispersive waves in a compressible
   Mooney-Rivlin rod},
   journal={Acta Mech.},
   volume={127},
   date={1998},
   number={1-4},
   pages={193--207},
}
\bib{Dan2001}{article}{
 	author={Danchin R.},
	title={A few remarks on Camassa-Holm equation},
	journal={Diff. Int. Equ.}
	volume={14},
	date={2001},
	pages={953--988}}
\bib{DP1}{article}{
   author={Degasperis A.},
   author={Procesi M.}, 
   title={ Asymptotic integrability, in Symmetry and Perturbation Theory}
   journal={Word Scientific}, 
  volume={211},
  pages={23--37}
   date={1999},
}
\bib{Dul}{article}{
  title={Camassa--Holm, Korteweg--de Vries  and other asymptotically equivalent equations for shallow water waves},
  author={Dullin H.},
  author={Gottwald G.},
  author={Holm, D.},
  journal={Fluid Dynamics Research},
  volume={33},
  number={1},
  pages={73--95},
  year={2003},
  publisher={Elsevier}
}	
 \bib{Dul}{article}{
  title={On asymptotically equivalent shallow water wave equations},
  author={Dullin H.}
   author={ Gottwald G.A.}, 
   author={Holm D.D.},
  journal={Physica D: Nonlinear Phenomena},
  volume={190},
  number={1},
  pages={1--14},
  year={2004},
  publisher={Elsevier}
}
\bib{Es1}{article}{
  title={Global weak solutions and blow-up structure for the Degasperis--Procesi equation},
  author={Escher J.},
  author={Liu Y.},
  author={Yin  Z.},
  journal={Journal of Functional Analysis},
  volume={241},
  number={2},
  pages={457--485},
  year={2006},
  publisher={Elsevier}
}
\bib{Es2}{article}{
  title={Well-posedness, blow-up phenomena, and global solutions for the b-equation},
  author={Escher J.},
  author={Yin Z.},
  journal={ J. Reine Angew. Math.},
  volume={624},
  number={1},
  pages={51--80},
  year={2008}}
\bib{Jor}{article}{
  title={The periodic b-equation and Euler equations on the circle},
 author={Escher J.},
  author={ Seiler J.},
  journal={Journal of Mathematical Physics},
 volume={51},
 number={5},
 year={2010}}
\bib{Fuch}{article}{
   author={Fuchssteiner B.},
   author={Fokas A.S},
   title={Symplectic structures, their B{\"a}cklund transformations and hereditary symmetries},
   journal={Physica D: Nonlinear Phenomena},
   volume={4},
   number={1}
  pages={299--303}
   year={1981}}  
  \bib{Gui}{article}{
  title={On the global existence and wave-breaking criteria for the two-component Camassa--Holm system},
  author={Gui G.}, 
  author={ Liu Y.},
  journal={Journal of Functional Analysis},
  volume={258},
  number={12},
  pages={4251--4278},
  year={2010},
  publisher={Elsevier}
}
\bib{He}{article}{
  title={Infinite propagation speed for the Degasperis--Procesi equation},
  author={Henry D.},
  journal={Journal of mathematical analysis and applications},
  volume={311},
  number={2},
  pages={755--759},
  year={2005},
  publisher={Elsevier}
}
\bib{Kat}{article}{
   author={Kato T.}, 
   title={Quasi-linear equations of evolution, with applications to partial differential equations},
   journal={Spectral Theory and Differential Equations, Lecture Notes in Math. Springer Verlag, Berlin}, 
   volume={448},
   date={1975},
      pages={25-27},  
}
\bib{Lel}{article}{
  title={Low-regularity solutions of the periodic Camassa--Holm equation},
  author={Lellis C.},
  author={ Kappeler T}, 
  author={Topalov  P.},
  journal={Communications in Partial Differential Equations},
  volume={32},
  number={1},
  pages={87--126},
  year={2007},
  publisher={Taylor \& Francis}}	 
\bib{LiOl}{article}{
   author={Li Y.A.},
   author={Olver P.J.}, 
   title={Well-posedness and blow-up solutions for an integrable nonlinearly dispersive model wave equation},
   journal={J. Diff. Eq.}, 
   volume={162},
   date={2000},
      pages={27--63},  
}
\bib{Liu}{article}{
  title={Global existence and blow-up phenomena for the Degasperis-Procesi equation},
  author={Liu Y.}, 
  author={Yin  Z.},
  journal={Communications in mathematical physics},
  volume={267},
  number={3},
  pages={801--820},
  year={2006},
  publisher={Springer}
}
\bib{Lun}{article}{
  title={Formation and dynamics of shock waves in the Degasperis-Procesi equation},
  author={Lundmark H.},
  journal={Journal of Nonlinear Science},
  volume={17},
  number={3},
  pages={169--198},
  year={2007},
  publisher={Springer}
}
\bib{Magri}{article}{
   author={Magri F.},  
   title={A Simple Model of the integrable Hamiltonian Equation,}
   journal={J.Math.Phys.}
   volume={19}
   date={1978}
   pages={1156-1162}
   }
   \bib{McKean04}{article}{
   author={McKean H.P.},
   title={Breakdown of the Camassa-Holm equation},
   journal={Comm. Pure Appl. Math.},
   volume={57},
   date={2004},
   number={3},
   pages={416--418}} 
  
\bib{Olver}{article}{
   title={On the Hamiltonian structure of evolution equations},
  author={Olver P.J.},
  booktitle={Mathematical Proceedings of the Cambridge Philosophical Society},
  volume={88},
  number={01},
  pages={71--88},
  year={1980},
  organization={Cambridge Univ Press}
}
 \bib{RB}{article}{
 author = {Rodr\'{\i}guez-Blanco G},
 title = {On the Cauchy problem for the Camassa-Holm equation},
 journal = {Nonlinear Anal.},
 volume = {46},
 number = {3},
 year = {2001},
 issn = {0362-546X},
 pages = {309--327},
 publisher = {Elsevier Science Ltd.},
 address = {Oxford, UK, UK}
 }
  \bib{Sa}{article}{
   author={Saha S.}, 
   title={Blow-Up results for the periodic peakon $b$-family of equations},
   journal={Comm.
Diff. and Diff. Eq.}, 
   volume={4},
  number={1},
pages={1--20},
 year={2013}}
 \bib{Yin1}{article}{
  title={On the Cauchy problem for an integrable equation with peakon solutions},
  author={Yin Zhaoyang},
  journal={Illinois Journal of Mathematics},
  volume={47},
  number={3},
  pages={649--666},
  year={2003},
  publisher={University of Illinois at Urbana-Champaign, Department of Mathematics}
}
\bib{Yin2}{article}{
  title={Global weak solutions for a new periodic integrable equation with peakon solutions},
  author={Yin Zhaoyang},
  journal={Journal of Functional Analysis},
  volume={212},
  number={1},
  pages={182--194},
  year={2004},
  publisher={Elsevier}}
  \bib{Gu}{article}{
  title={Blow-up and global solutions to a new integrable model with two components},
  author={ Zhengguang Guo},
  journal={Journal of Mathematical Analysis and Applications},
  volume={372},
  number={1},
  pages={316--327},
  year={2010},
  publisher={Elsevier}}
\bib{Zh}{article}{
  title={Blow-up phenomenon for the integrable Degasperis--Procesi equation},
  author={Zhou, Yong},
  journal={Physics Letters A},
  volume={328},
  number={2},
  pages={157--162},
  year={2004},
  publisher={Elsevier}
}
\bib{Zh1}{article}{
  title={On solutions to the Holm--Staley $b$-family of equations},
  author={Zhou, Yong},
  journal={Nonlinearity},
  volume={23},
  number={2},
  pages={369--381},
  year={2010},
}  
\end{thebibliography}
\end{document}